\documentclass[11pt,a4paper]{amsart}
\usepackage{amsmath}
\usepackage{amsfonts}
\usepackage{amssymb}
\usepackage{amsthm}
\usepackage{color}
\usepackage[parfill]{parskip}
\usepackage{etoolbox}

\usepackage[dvips]{graphics}
\usepackage[dvips]{graphicx}


\newtheorem{theorem}{Theorem}[section]
\newtheorem{corollary}[theorem]{Corollary}
\newtheorem{lemma}[theorem]{Lemma}
\newtheorem{prop}[theorem]{Proposition}
\newtheorem{defn}[theorem]{Definition}
\newtheorem{definition}[theorem]{Definition}
\newtheorem{eg}[theorem]{Example}
\newtheorem{example}[theorem]{Example}
\newtheorem{remark}[theorem]{Remark}

\newcommand{\op}{\overline{\Bbb P}_n}
\newcommand{\up}{\underline{\Bbb P}_n}
\newcommand{\ox}{\overline{X}}
\newcommand{\ux}{\underline{X}}

\newcommand{\ext}{{\it {E}}}
\newcommand{\st}{{\rm {St}}}
\newcommand{\St}{{\rm {St}}}
\newcommand{\lk}{{\rm {Lk}}}
\newcommand{\su}{\subseteq}
\newcommand{\pa}{\partial}
\newcommand{\Z}{{\Bbb Z}}

\newcommand{\E}{{\Bbb E}}
\newcommand{\U}{{\mathcal U}}
\newcommand{\R}{\Bbb R}
\newcommand{\D}{\mathfrak{D}}

\begin{document}

\title{Random simplicial complexes, duality and the critical dimension}

\author
{Michael Farber} 
\thanks{Michael Farber was partially supported by the EPSRC, by the IIAS and by the Marie Curie Actions, FP7, in the frame of the EURIAS Fellowship Programme. }
\address{Queen Mary University of London}
\author{Lewis Mead} \thanks{Lewis Mead was supported by an EPSRC PhD fellowship.}
\address{Queen Mary University of London}
\author{Tahl Nowik}
\address{Bar Ilan University}

\begin{abstract} In this paper we discuss two general models of random simplicial complexes which we call {\it the lower and the upper} models. We show that these models are dual to each other with respect to combinatorial Alexander duality. 
The behaviour of the Betti numbers in the lower model is characterised by the notion of {\it critical dimension}, which was introduced by A. Costa and M. Farber in \cite{farber4}: random simplicial complexes in the lower model are homologically approximated by a wedge of spheres of dimension equal the critical dimension. 
In this paper we study the Betti numbers in the upper model and introduce new notions of {\it critical dimension and spread}. 
We prove that (under certain conditions) an upper random simplicial complex is homologically approximated by a wedge of spheres of the critical dimension.

\end{abstract}

\maketitle

\section{Introduction}
The study of random simplicial complexes and random manifolds is motivated by potential applications to modelling of large complex systems in engineering and computer science applications. 
Random topological objects can also be used in pure mathematics for constructing examples of objects with 
rare combinations of topological properties. 

Several models of random manifolds and random simplicial complexes were suggested and studied recently. One may mention random surfaces \cite{PS}, random 3-dimensional manifolds
\cite{DT}, random closed smooth high-dimensional manifolds appearing as configuration spaces  \cite{F};  see \cite{Ksurvey} for a survey.
 
In the present paper we study two very general probabilistic models generating  random simplicial complexes of arbitrary dimension which we call {\it the lower and upper models}.

In the case of the lower model one builds the random simplicial complex inductively, step by step, starting with a random set of vertices, 
then adding randomly edges between the selected vertices, and on the following step adding randomly 2-simplexes (triangles) to the random graph obtained on the previous stage, and so on. 
Examples are given by the Erd\H{o}s--R\'enyi \cite{ER} random graphs and their high dimensional generalizations,
the Linial, Meshulam, Wallach \cite{linmesh} , \cite{walmesh} random simplicial complexes, as well as random clique complexes \cite{Kahle1}, \cite{Kahle3}. In larger generality the lower model of random simplicial complexes was studied in a series of papers \cite{farber}, \cite{farber2}, \cite{farber3}, \cite{farber4} under the name of {\it multiparameter random simplicial complexes}; the name reflects the fact that the geometric and topological properties of simplicial complexes in this model depend on the set of probability parameters $p_0, p_1, \dots, p_r$ reflecting probabilities with which simplexes of various dimensions are included. 

In the case of the upper model one selects randomly a set of simplexes of various dimensions and then adds all their faces to obtain a random simplicial complex. The upper model was also studied recently by a variety of authors, see for example \cite{CO} and references therein. 

In this paper we show that the upper and lower models are dual to each other. More precisely, the upper random simplicial complex 
is homotopy equivalent to the complement of the lower random simplicial complex in the $(n-1)$-dimensional sphere $\partial \Delta_n$, i.e. we are dealing here with the Alexander duality. Under the duality correspondence the probability parameters $p_\sigma$ should be replaced by $q_{\hat\sigma}=1-p_{\hat\sigma}$ where $\hat\sigma$ is the simplex spanned by the complement of the set of vertexes of $\sigma$. 
We see that the duality matches {\it a sparse} lower model (when $p_\sigma\to 0$) with {\it a dense} upper model (when $p_\sigma\to 1)$ and vice versa. 

In a recent paper \cite{farber4} the authors established an interesting pattern of behaviour of the Betti numbers of random simplicial complexes in the lower model. It was shown that there exists a specific dimension $k$ (called {\it the critical dimension}) such that the Betti number $b_k(Y)$  
is large, the Betti numbers $b_j(Y)$ vanish for $0<j<k$ and are significantly smaller than $b_k(Y)$ for $j>k$. In other words, a random simplicial complex in the lower model can be approximated (homologically) by a wedge of spheres of the critical dimension. 

One of the goals of this paper is to investigate the Betti numbers of random simplicial complexes in the upper model. 
We define the notions of {\it the critical dimension} $k^\ast$ and {\it the spread}  $s$ and show that the exponential growth rate of the face numbers $f_\ell(Y)$ is maximal and constant in dimensions $\ell$ satisfying $k^\ast\le \ell\le k^\ast+s$. In the case when the spread is zero 
$s=0$ we show that the critical dimension $k^\ast$ behaves similarly to the lower model: the Betti number $b_{k^\ast}(Y)$ is large and maximal, the Betti numbers $b_j(Y)$ vanish for $0<j<k^\ast$ and $b_j(Y)$ is significantly smaller than $b_{k^\ast}(Y)$ for $j>k^\ast$. 

Preprint \cite{FM} continues to explore the technique of the present paper based on Alexander duality
and focuses on the lower and upper models in {\it the medial regime}, i.e. in the situation when the probability parameters $p_\sigma$ are allowed to approach neither $0$ nor $1$. 
 The results of \cite{FM} show that non-vanishing Betti numbers in the medial regime may occur in a narrow ranges of dimensions which are described.

In this paper we use terminology according to which simplicial complexes are both combinatorial and geometric objects;
this should lead neither to misunderstanding nor to ambiguity. 
Thus, in combinatorial topology, a simplex is a nonempty finite set of points (vertices) and in geometric topology a simplex is a topological space homeomorphic to the convex hull of a finite set of points in general position in the Euclidean space.

\section{Random Hypergraphs}\label{d}
The symbol $[n]$ stands for  the set $ \{0, 1,\dots,n\}$.

We shall consider hypergraphs $X$ with vertex sets contained in $[n]$; each such hypergraph $X$ 
is a collection 
of non-empty subsets  $\sigma\subseteq [n]$. 
We shall denote by 
$\Omega_n$ the set of all such hypergraphs. 

Next we define a probability measure on $\Omega_n$.  
Let $$p_\sigma\in [0,1]$$ be a probability parameter associated with each non-empty subset $\sigma\subseteq [n]$. 
Using these parameters $p_\sigma$ we may define a probability function $\mathbb P_n$ on $\Omega_n$ by the formula
\begin{eqnarray}\label{prob0} \mathbb{P}_n(X) =\prod_{\sigma\in X}p_\sigma\cdot\prod_{\sigma\not\in X} q_\sigma.\end{eqnarray}
Here $q_\sigma$ denotes $ 1- p_\sigma$. Formula (\ref{prob0}) can be described by saying that each simplex $\sigma\subseteq [n]$ is included into a random hypergraph $X$ 
with probability $p_\sigma$ 
independently of all other simplexes. Essentially $\mathbb{P}_n$ is a Bernouilli measure on the set of all non-empty subsets of $[n]$.

\section{The upper  and lower models of random simplicial complexes}\label{sec:twomodels}

\subsection{} Let $\Omega^\ast_n\subseteq \Omega_n$ denote the set of all simplicial complexes on the vertex set $[n]=\{0, 1, \dots, n\}$. Recall that a hypergraph $X$ is a simplicial complex if it is closed with respect to taking faces, i.e. if $\sigma\in X$ and $\tau\subseteq \sigma$ imply that 
$\tau\in X$. 

Let $\Delta_n$ denote the simplicial complex consisting of all non-empty subsets of $[n]$. The complex $\Delta_n$ is known as the $n$-dimensional simplex spanned by the set $[n]$. The set $\Omega_n^\ast$ is the set of all subcomplexes of $\Delta_n$. 

There are two natural surjective maps which are the identity on $\Omega_n^\ast$ (in other words, they are retractions) 
\begin{eqnarray}\label{maps}\overline \mu, \, \, \underline \mu: \, \Omega_n\to \Omega^\ast_n\end{eqnarray}
which are defined as follows. 

For a hypergraph $X\in \Omega_n$ we denote by 
$\overline \mu(X) =\overline X$ the smallest simplicial complex in $\Omega^\ast_n$ containing $X$. A simplex $\tau\in \Delta_n$ belongs to $\overline X$ if and only if for some $\sigma\in X$ one has $\sigma\supseteq \tau$. 

On the other hand, the simplicial complex $\underline \mu(X) =\underline X$ is the largest simplicial complex in $\Omega^\ast_n$ contained in $X$.
A simplex $\tau\subseteq [n]$ belongs to $\underline X$ if and only if every simplex $\sigma\subseteq \tau$ belongs to $X$. 

One has 
\begin{eqnarray}
\underline X \subseteq X \subseteq \overline X.
\end{eqnarray}

We shall denote by 
\begin{eqnarray}\label{measures}
\overline {\Bbb P}_n= \overline \mu_\ast({\Bbb P}_n)\quad \mbox{and}\quad \underline {\Bbb P}_n= \underline \mu_\ast({\Bbb P}_n)
\end{eqnarray}
the two probability measures on the space of simplicial complexes $\Omega^\ast_n$ obtained as the push-forwards (or image measures) of the measure (\ref{prob0})
with respect to the maps (\ref{maps}). Explicitly, for a simplicial complex $Y\subseteq \Delta_n$ one has
$$\op(Y) =\sum_{X\in \Omega_n, \, \overline X=Y} \mathbb P_n(X) \quad \mbox{and}\quad \up(Y) =\sum_{X\in \Omega_n, \,\underline X=Y} \mathbb P_n(X).$$


Our goal in this paper is to compare the properties of random simplicial complexes with respect to the two measures (\ref{measures}).

Next we make the following remarks. 
\begin{remark}\label{rm11}{\rm 
If $p_\sigma=0$ for some $\sigma$ then a random hypergraph $X$ contains $\sigma$ with probability 0, hence we see that 
$\sigma\not\in \underline X$ with probability 1. Thus, if $p_\sigma=0$, the lower measure $\up$ is supported on the set of simplicial subcomplexes 
$Y\subseteq \Delta_n-\St(\sigma)$. Moreover, if $p_\tau=0$ for every simplex $\tau\supseteq \sigma$ then $\sigma\not\in \overline X$ with probability one and the measure $\op$ is supported on the set of simplicial subcomplexes 
$Y\subseteq \Delta_n-\St(\sigma)$. The symbol $\St(\sigma)$ denotes the star of the simplex $\sigma$, i.e. the set of all simplexes containing $\sigma$. 
}
\end{remark}

\begin{remark}\label{rm1}
{\rm 
Consider now the opposite extreme, $p_\sigma=1$. Then a random hypergraph $X$ contains $\sigma$ with probability 1. This implies 
that $\sigma \in \overline X$ with probability 1. Moreover, if $p_\tau=1$ for every $\tau\subseteq \sigma$ then $\sigma\in \underline X$ with probability 1 and 
the measure $\up$ is supported on the set of simplicial complexes $Y\subseteq \Delta_n$ containing $\sigma$. 

Later (see Corollary \ref{cor57}) we shall establish the following explicit formulae. For a simplicial subcomplex $Y\subseteq \Delta_n$ one has
\begin{equation}
\up(Y) = \prod_{\sigma \in Y} p_\sigma \cdot \prod_{\sigma\in E(Y)} q_\sigma, \quad \mbox{and}\quad
\op(Y) = \prod_{\sigma \in M(Y)} p_\sigma \cdot \prod_{\sigma\not\in Y} q_\sigma,
\end{equation}
where the symbols $E(Y)$ and $M(Y)$ are defined as follows:
}
\end{remark}
\begin{definition}
For a simplicial subcomplex $Y\subseteq \Delta_n$ we denote by $E(Y)$ the set of {\it external simplexes}, i.e. simplexes $\sigma\in \Delta_n$ such that $\sigma\not\in Y$ but the boundary $\partial \sigma$ is contained in $Y$. Besides, the symbol $M(Y)$ denotes the set of {\it maximal simplexes} of $Y$, i.e.\ those which are not proper faces of other simplexes of $Y$. 
\end{definition}

%

\section{Duality between the upper and lower models}\label{du}

In this section we present duality between the upper and lower models; this theme will continue in \S \ref{ad} where we shall show that the the simplicial complexes produced by the upper and lower models are Alexander dual to each other, and moreover, one is homotopy equivalent to the complement of the other in the ambient sphere $\pa\Delta_n$. 

In this section our emphasis is purely topological--combinatorial and probabilistic considerations appear only briefly in Proposition \ref{pn}.

Recall that $\pa \Delta_n$ is the simplicial complex with vertex set 
$[n]=\{0, 1, \dots, n\}$ in 
which simplexes are all nonempty subsets $V\subset [n]$, except $V=[n]$. Clearly the geometric realisation of $\pa\Delta_n$ is homeomorphic to sphere of dimension $n-1$.

For a simplex  $\sigma \in \pa\Delta_n$ we define $\hat{\sigma}$ to be the simplex $[n]-\sigma$. For a hypergraph $X\su \pa\Delta_n$  we denote
by $i(X)$ the image of $X$ under the map $\sigma \mapsto \hat{\sigma}$, i.e. $i(X)=\{\hat{\sigma} : \sigma\in X \}$. Since  $\sigma \mapsto \hat{\sigma}$ is an involution, also $i$ is an involution.  
We have $$i(X\cap Y)=i(X)\cap i(Y), \quad i(X\cup Y)=i(X)\cup i(Y),$$ and $X \su Y$ if and only if $i(X) \su i(Y)$.

Since $\sigma \su \tau$ if and only if $\hat{\sigma} \supseteq \hat{\tau}$,  we have that a hypergraph $X$ is a simplicial complex if and only if $i(X)$ is an ``anti-complex'', by which we mean that if $\sigma \in i(X)$, and $\tau \supseteq \sigma$ then $\tau \in i(X)$.

A second involution on the set of hypergraphs is the map 
$$j(X)=X^c=\{\sigma \in \pa\Delta_n : \sigma \not\in X \}.$$
We have $X \su Y$ if and only if $j(X) \supseteq j(Y)$, and by De Morgan's rules we have 
$$j(X\cap Y)=j(X)\cup j(Y),\quad j(X\cup Y)=j(X)\cap j(Y).$$ Again, we have  that $X$ is a simplicial complex if and only if $j(X)$ is an anti-complex.
Since $\sigma\mapsto \hat\sigma$ is a bijection we have $i(X^c)=(i(X))^c$ which means $i\circ j =j\circ i$, and so $i\circ j$ is again an involution. Finally, for a hypergraph $X\subset \pa\Delta_n$ we define the \emph{dual} hypergraph
 $$c(X)=i\circ j(X).$$

Combining the properties of $i$ and $j$ mentioned above we get
the following properties of $c(X)$.

\begin{lemma}\label{pr} For hypergraphs $X,Y\subseteq \pa\Delta_n$ we have:
\begin{enumerate} 
\item $\sigma \in X$ if and only if $\hat{\sigma} \not\in c(X)$.
\item $c(c(X))=X$. 
\item $X\su Y$ if and only if $c(X)\supseteq c(Y)$.
\item $c(X\cap Y)=c(X)\cup c(Y)$ and $c(X\cup Y)=c(X)\cap c(Y)$.
\item $X$ is a simplicial complex if and only if $c(X)$ is a simplicial complex.
\end{enumerate}
\end{lemma}

The complex $c(X)$ is sometimes known as the Bj\"orner--Tanner dual of a simplicial complex $X$, see \cite{BT} and also \S \ref{ad}. 

\begin{lemma}\label{ch}
For every hypergraph $X\subseteq \pa\Delta_n$ we have 
$c(\overline{X})=\underline{c(X)}$
and similarly $c(\underline{X}) =\overline{c(X)}$.
\end{lemma}

\begin{proof} Since $\underline X\subseteq X\subseteq \overline X$ we have $c(\overline X) \subseteq c(X)\subseteq c(\underline X)$ and hence
\begin{eqnarray}\label{long}
c(\overline X) \subseteq \underline{c(X)} \subseteq c(X)\subseteq \overline{c(X)}\subseteq c(\underline X),
\end{eqnarray} 
using properties (3) and (5). Applying the operator $c$ to the inclusion $c(\overline X) \subseteq \underline{c(X)}$ and replacing $X$ by $c(X)$ we get $\overline{c(X)} \supseteq c(\underline X)$ which is the inverse to the right inclusion in (\ref{long}). Thus, 
$\overline{c(X)} = c(\underline X)$. Replacing here $X$ by $c(X)$ and applying the operator $c$ to both sides we obtain $c(\overline{X})=\underline{c(X)}$. 
%
%
%

%
%
\end{proof}

\begin{prop}\label{pn}
Given ${\Bbb P}_n$ defined on $\Omega_n$ by probabilities $\{p_\sigma\}_{\sigma \in \pa\Delta_n}$, define a new probability measure ${\Bbb P}'_n$ on $\Omega_n$ by probabilities $\{p'_\sigma\}_{\sigma \in \pa\Delta_n}$ where $$p'_\sigma=q_{\hat{\sigma}}=1-p_{\hat{\sigma}}.$$ 
Then
\begin{enumerate}
\item For every hypergraph $X\subseteq \pa\Delta_n$, $$\mathbb{P}_n(c(X)) = \mathbb{P}'_n(X).$$
\item For every simplicial complex $Y\subseteq \pa\Delta_n$,
 $$\op(c(Y))=\up'(Y)\quad \mbox{and}\quad \up(c(Y))=\op'(Y).$$
\end{enumerate}
\end{prop}

\begin{proof}
(1) By definition of $\mathbb{P}_n$ and by Lemma \ref{pr}(1) we have
\begin{equation*}\begin{split}
\mathbb{P}_n(c(X)) =\prod_{\sigma\in c(X)}p_\sigma\cdot\prod_{\sigma\not\in c(X)} q_\sigma
=\prod_{\hat{\sigma}\not\in X}p_\sigma\cdot\prod_{\hat{\sigma}\in X} q_\sigma
=\prod_{\sigma\not\in X} p_{\hat{\sigma}} \cdot\prod_{\sigma\in X} q_{\hat{\sigma}} \\
=\prod_{\sigma\not\in X} q'_{\sigma} \cdot\prod_{\sigma\in X} p'_{\sigma}=\mathbb{P}'_n(X).
\end{split}\end{equation*}

(2) By (1) and by Lemma \ref{ch}, for every simplicial complex $Y$, 
$$\op(c(Y))=\sum_{\overline X=c(Y)} \mathbb P_n(X)=\sum_{c(\overline X)=Y} \mathbb P'_n(c(X))
=\sum_{\underline{c(X)}=Y} \mathbb P'_n(c(X))=\up'(Y).$$
\end{proof}


\begin{lemma}\label{me}
If $Y\subseteq\pa\Delta_n$ is a simplicial complex then a simplex $\sigma$ is an external simplex of $Y$ if and only if $\hat{\sigma}$ is a maximal simplex of $c(Y)$, and vice versa. 
\end{lemma}

\begin{proof}
An external simplex of $Y$ is by definition a \emph{minimal} simplex \emph{not} in $Y$.
Thus the statement follows from Lemma \ref{pr}(1) and the fact that $\sigma \su \tau$ if and only if $\hat{\sigma} \supseteq \hat{\tau}$.
\end{proof}

%
%
%
%
%
%

\section{ The Sandwich Formulae}


\subsection{}
Let $A\su B\su \pa\Delta_n$ be two simplicial complexes. 
In both the lower and upper probability measures $\up$ and $\op$,
we ask what is the probability that a random simplicial complex $Y$ 
satisfies $A \su Y \su B$. 
That is, we are interested in finding the probability 
\begin{eqnarray}\label{sup}
\up(A\su Y\su B) = \sum_{A\su Y\su B}\up(Y) = \sum_{A\subseteq \ux \subseteq B} {\Bbb P}_n(X). 
\end{eqnarray}
Here $Y$ denotes a simplicial subcomplex $Y\in\Omega^\ast_n$ and $X$ denotes a hypergraph $X\in \Omega_n$. 
Similarly, we want to calculate explicitly the quantities
\begin{eqnarray}\label{sop}
\op(A\su Y\su B) = \sum_{A\su Y\su B}\op(Y) = \sum_{A\su \ox \su B} {\Bbb P}_n(X). 
\end{eqnarray}

\subsection{} Note that for hypergraphs the answer to a similar question is simple: 
$${\Bbb P}_n(A\subseteq X\subseteq B) = \prod_{\sigma\in A}p_\sigma\cdot \prod_{\sigma\not\in B} q_\sigma.$$
Here $A$, $B$ are fixed hypergraphs and $X$ is a random hypergraph.

\subsection{} Recall that for a simplicial complex $B$, the symbol $E(B)$ denotes the set of all external simplices of $B$, 
i.e. simplices $\sigma\in \pa \Delta_n$ such that $\sigma \not\in B$ but $\partial\sigma \su B$. 

\begin{prop}[Sandwich formula for the lower model]\label{Lsand}
Let $A\subseteq B\subseteq \pa\Delta_n$ be two simplicial complexes. For every subset $S\su E(B)$ let 
$A_S$ be the set of all simplices $\tau \not\in A$ such that $\tau \su \sigma$ for some $\sigma \in S$. 
Let $$P_S =\prod_{\tau \in A_S}p_\tau$$  and  
$$\tilde{P}=\prod_{\tau\in A} p_\tau.$$
Then 
\begin{eqnarray}\label{lowersandwich}
\up(A\su Y\su B) = \tilde{P}\cdot \sum_{S\su E(B)}(-1)^{|S|}P_S,
\end{eqnarray}
where by definition $P_\varnothing=1$. 
\end{prop}

\begin{proof}Since $A$ and $B$ are simplicial complexes,
a hypergraph $X$ satisfies $A\su\ux\su B$ if and only if $X\supseteq A$ and $X\not\supseteq A_{\{\sigma\}}$ for all $\sigma \in E(B)$. So we have 
\begin{equation*}\begin{split}
\{X  : A \su \ux \su B \} =\{X:X\supseteq A\}\cap \bigcap_{\sigma \in E(B)} \{X  : X\not\supseteq A_{\{\sigma\}}\} \\ = \bigcap_{\sigma \in E(B)} \{X  :  X \supseteq A , X\not\supseteq A_{\{\sigma\}}\}.
\end{split}\end{equation*}
To evaluate the probability of this event we use the inclusion-exclusion formula 
with ambient set $\{X  :  X \supseteq A \}$, so the event 
$\{X  :  X \supseteq A , X\not\supseteq A_{\{\sigma\}}\}$ is the complement of the event 
$\{X  :  X \supseteq A  , X\supseteq A_{\{\sigma\}}\}=\{X  :   X\supseteq A \cup A_{\{\sigma\}}\}$.
We thus get  
\begin{equation*}\begin{split}
{\Bbb P}_n (A \su \ux \su B)   
&= \sum_{S \su E(B)}  (-1)^{|S|}{\Bbb P}_n\Big( \bigcap_{\sigma \in S} \{X  :  X\supseteq A \cup A_{\{\sigma\}}\}\Big)  \\
 &= \sum_{S \su E(B)}  (-1)^{|S|}{\Bbb P}_n\big( X\supseteq A \cup A_S\big)
\\ &= \sum_{S\su E(B)}(-1)^{|S|} \prod_{\tau \in A \cup A_S}p_\tau  
=  \sum_{S\su E(B)}(-1)^{|S|}\tilde{P} P_S.\end{split}\end{equation*}  
The second equality holds 
since  $A_S = \bigcup_{\sigma \in S} A_{\{\sigma\}}$. 
\end{proof}

Using the duality introduced in Section \ref{du}, we obtain the following dual result for $\op$.
Recall that for a simplicial complex $A$, the symbol $M(A)$ denotes the set of maximal simplices in $A$.

\begin{prop}[Sandwich formula for the upper model]\label{Usand} 
Let $A\subseteq B\subseteq \pa\Delta_n$ be two simplicial complexes. For every $S\su M(A)$ let 
$B_S$ be the set of all simplices $\tau \in B$ such that $\tau \supseteq \sigma$ for some $\sigma \in S$. 
Let $$Q_S =\prod_{\tau \in B_S}q_\tau$$  and  
$$\tilde{Q}=\prod_{\tau\not\in B} q_\tau.$$
Then 
\begin{eqnarray}\label{uppersandwich}
\op(A\su Y\su B) = \tilde{Q}\cdot \sum_{S\su M(A)}(-1)^{|S|}Q_S,
\end{eqnarray}
where by definition $Q_\varnothing=1$. 
\end{prop}

\begin{proof}
This follows from Proposition \ref{Lsand} via the dual measure ${\Bbb P}'_n$ using Proposition \ref{pn} and Lemma \ref{me}.
\end{proof}

As a corollary to the proof of Proposition \ref{Lsand}, we get the following characterization of the probability measures $\up$ and $\op$.

\begin{corollary}[Intrinsic characterisation of the upper and lower measures]\label{chr} Let $\lambda$ be a probability measure on the set of simplicial complexes $Y \su \pa\Delta_n$. Let $\{p_\sigma\}_{\sigma\in\pa\Delta_n}$ be a fixed assignment of numbers $0 \leq p_\sigma \leq 1$, and denote $q_\sigma=1-p_\sigma$.
\begin{enumerate}
\item We have $\lambda=\up$ if and only if for every simplicial complex $K$, $\lambda(Y\supseteq K) = \prod_{\sigma\in K}p_\sigma$.
\item We have $\lambda=\op$ if and only if for every simplicial complex $K$,  $\lambda(Y\su K) = \prod_{\sigma\not\in K}q_\sigma$,

\end{enumerate}
\end{corollary}

\begin{proof}
The ``only if'' direction follows immediately from the definition of the lower and upper models, and is also
 a special case of Propositions \ref{Lsand}, \ref{Usand}. We show the ``if'' direction of (1) and (2) as follows.

 (1) The only place in the proof of Proposition \ref{Lsand} where the probability measure was used was in the equality
${\Bbb P}_n\big( X\supseteq A \cup A_S\big)= \prod_{\sigma \in A \cup A_S}p_\sigma$. We note however that $K=A\cup A_S$ is a \emph{simplicial complex}, and so $X\supseteq K$ if and only if $\ux \supseteq K$, so 
${\Bbb P}_n( X\supseteq K) ={\Bbb P}_n( \ux\supseteq K) = \up(Y \supseteq K)$. So in fact we only needed to know that $\up(Y \supseteq K)= \prod_{\sigma \in K}p_\sigma$.

(2) Define another probability measure $\lambda'$ on simplicial complexes by $\lambda'(Y)=\lambda(c(Y))$.
Then for every simplicial complex $K$, $\lambda'(Y\supseteq K)=\lambda(Y \su c(K))
= \prod_{\sigma\not\in c(K)}q_\sigma= \prod_{\hat{\sigma}\in K}q_\sigma
= \prod_{\sigma\in K}q_{\hat{\sigma}}= \prod_{\sigma\in K}p'_{\sigma}$, where as in Proposition \ref{pn} we define 
$p'_{\sigma}=q_{\hat{\sigma}}$ and the corresponding $\up'$. By (1) applied to $\{p'_\sigma\}_{\sigma\in\pa\Delta_n}$ we get $\lambda'=\up'$ and
so by Proposition \ref{pn}, $\lambda= \op$.
\end{proof}

Next we consider a few special cases where simplified sandwich formulae hold. 

\begin{corollary}\label{ee} 
Let $A \su B \su \pa\Delta_n$ be two simplicial complexes.
\begin{enumerate}
\item In the notation of Proposition \ref{Lsand}, if the sets $A_{\{\sigma\}}$ for $\sigma \in E(B)$ are disjoint, then 
$$\up(A\su Y\su B) = \tilde{P}\cdot  \prod_{\sigma\in E(B)} (1-P_{\{ \sigma \}}).$$
\item In the notation of Proposition \ref{Usand}, if the sets $B_{\{\sigma\}}$ for $\sigma \in M(A)$ are disjoint, then 
$$\op(A\su Y\su B) = \tilde{Q}\cdot  \prod_{\sigma\in M(A)} (1-Q_{\{ \sigma \}}).$$
\end{enumerate}
\end{corollary}

\begin{proof}
(1) Since the sets $A_{\{\sigma\}}$ are disjoint, for every $S \su E(B)$ we have 
$P_S=\prod_{\sigma \in S}P_{\{\sigma\}}$. Thus   
$$\prod_{\sigma\in E(B)} (1-P_{\{ \sigma \}})
= \sum_{S\su E(B)}(-1)^{|S|} \prod_{\sigma \in S}P_{\{\sigma\}}  
=  \sum_{S\su E(B)}(-1)^{|S|} P_S.$$ The statement (2) is similar.
\end{proof}

\begin{example}\label{ex3}
{\rm
For a simplex $\sigma$ we have $$\op(\sigma\in Y) = 1-Q_{\{ \sigma \}}
=1-\prod_{\tau \supseteq \sigma}q_\tau.$$
This may be seen from Corollary \ref{ee}(2) taking $A$ to be $\sigma$ as a simplicial complex, i.e. 
$A=\{\tau : \tau \su \sigma\}$, having  $M(A)=\{\sigma\}$. 

It also follows immediately from the definition of the upper model that 
$\op(\sigma\not\in Y) =\prod_{\tau \supseteq \sigma}q_\tau.$

}
\end{example}

\begin{corollary}\label{cor8}
Let $A \su B \su \pa\Delta_n$ be two simplicial complexes.
\begin{enumerate}
\item  If $E(B)\su E(A)$ then 
$\up(A\su Y\su B) = \prod_{\sigma\in A}p_\sigma\cdot \prod_{\sigma\in E(B)} q_\sigma$.
\item If $M(A) \su M(B)$ then 
$\op(A\su Y\su B) = \prod_{\sigma\not\in B} q_\sigma \cdot \prod_{\sigma\in M(A)} p_\sigma $.
\end{enumerate}
\end{corollary}

\begin{proof}
(1) Since $E(B)\su E(A)$ we have for every $\sigma \in E(B)$, $A_{\{\sigma\}}=\{\sigma\}$.
Thus $P_{\{\sigma\}}=p_\sigma$, and so the factors $1-P_{\{\sigma\}}$ of Corollary  \ref{ee} reduce to
$1-p_\sigma=q_\sigma$. (2) is similar.
\end{proof}

Finally we also obtain an explicit formula for $\up$ and $\op$ themselves:

\begin{corollary}\label{cor57} Let $Y\subseteq \pa\Delta_n$ be a simplicial complex. Then
\begin{enumerate}
\item
$\up(Y)=\prod_{\sigma\in Y}p_\sigma \cdot\prod_{\sigma\in E(Y)} q_\sigma$.
\item
$\op(Y)=\prod_{\sigma\not\in Y} q_\sigma  \cdot \prod_{\tau\in M(Y)} p_\sigma $.
\end{enumerate}
\end{corollary}
\begin{proof} Apply Corollary \ref{cor8} with $A=B=Y$. 
\end{proof}

\section{Links as random complexes}

\subsection{Links in the upper model}\label{sec61} Consider random simplicial complexes $Y$ containing a fixed simplex $\sigma\subset [n]$. 
The link of $\sigma$ in $Y$, 
$$\lk_Y(\sigma)=L\subseteq \Delta',$$ 
is a random simplicial subcomplex of the simplex $\Delta'$, where  $\Delta'$ denotes the simplex
spanned by the vertexes $[n]-\sigma$. Recall that by the definition a simplex $\tau\in \Delta'$ lies in the link $\lk_Y(\sigma)$ if and only if the simplex $\sigma\tau$ belongs to $Y$. Here $\sigma\tau$ denotes the simplex $\sigma\cup \tau$ which geometrically is represented by the join $\sigma\tau=\sigma\ast\tau$.

Below in this section we shall consider the probability measures on the set of simplicial subcomplexes of $\Delta'$ which arise as the push-forwards of the conditional probability measures $$\frac{\op(Y)}{\op(\sigma\in Y)} \quad \mbox{and} \quad\frac{\up(Y)}{\up(\sigma\in Y)}$$ under the map $Y\mapsto \lk_Y(\sigma)$. These two measures will be denoted by $\overline\lambda$ and $\underline\lambda$ 
correspondingly.

\begin{theorem}\label{linksimplex} Let $Y\subseteq \Delta_n$ be a random simplicial complex distributed with respect to the upper measure $\op$ 
with the set of probability parameters $p_\sigma$. Assume that $Y$ contains a fixed simplex $\sigma\in \Delta_n$. 
Then $\overline \lambda$ (defined above) equals
\begin{eqnarray}\label{anomaly1}
 c_\sigma\cdot \overline {\Bbb P}'+(1-c_\sigma)\cdot \lambda_\varnothing,
\end{eqnarray}
where $\overline {\Bbb P}'$ denotes the upper probability measure on subcomplexes of $\Delta'$ 
with the set of probability parameters 
$p'_\tau=p_{\sigma\tau}.$ 
\end{theorem}
The symbol $\lambda_\varnothing$ in (\ref{anomaly1}) denotes the measure which is supported on the empty subcomplex, i.e. 
$$\lambda_\varnothing(L) = \left\{\begin{array}{lll}
1, & \mbox{for} & L=\varnothing,\\ \\
0, & \mbox{otherwise}.&
\end{array}
\right. 
$$
The symbol $c_\sigma$ in (\ref{anomaly1}) stands for 
$$c_\sigma = \left(1- \prod_{\tau\supseteq \sigma} q_{\tau}\right)^{-1}=\, \op(\sigma\in Y)^{-1}\ge 1,$$
see Example \ref{ex3}.

\begin{proof} 
We have
$$\overline\lambda(L) = \frac{\op(\sigma\in  Y\, \, \& \, \, \lk_Y(\sigma)=L)}{\op(\sigma\in  Y)}= 
\frac{\op(\sigma\ast L \subseteq  Y\subseteq \sigma\ast L\cup (\partial \sigma\ast\Delta'))}{\op(\sigma\in Y)}
$$
Assuming that $L\not=\varnothing$ we see that the maximal simplexes of $\sigma\ast L$ are of the form $\sigma \ast \tau=\sigma\tau$ where $\tau$ is a maximal simplex of $L$. These are also maximal simplexes of $\sigma\ast L\cup \partial \sigma \ast \Delta'$. Hence applying Corollary \ref{cor8}(2) we find (assuming that 
$L\not=\varnothing$)
$$\overline{\lambda}(L) = c_\sigma\cdot \prod_{\tau\in M(L)}p_{\sigma\tau}\cdot \prod_{  \tau\in \Delta'-L}q_{\sigma\tau}=c_\sigma\cdot \prod_{\tau\in M(L)} p'_\tau\cdot \prod_{\tau\in \Delta'-L}q'_\tau = c_\sigma \cdot \overline {\Bbb P}'(L),$$
where 
$ c_\sigma= \op(\sigma\in Y)^{-1}$. 
Besides, for $L=\varnothing$ we have 
$$\overline{\lambda}(\varnothing) = \frac{\op(\sigma\subseteq Y\subseteq \sigma\cup (\partial \sigma \ast\Delta'))}{\op(\sigma\in Y)} = c_\sigma  p_\sigma\prod_{\tau\in \Delta'} q_{\sigma\tau}= c_\sigma p_\sigma\op'(\varnothing).$$
Thus, noting $c_\sigma =(1-q_\sigma\op'(\varnothing))^{-1}$,
we obtain (\ref{anomaly1}).
\end{proof}

Note that $\overline\lambda$ is an upper type probability measure with anomaly at $\varnothing$.

\subsection{Links in the lower model} Next we describe the measure $\underline \lambda$ as defined in \S \ref{sec61}. 
It is the push-forward of the conditional probability measure on the set of simplicial complexes $Y\subset \Delta_n$ containing a given simplex $\sigma$ with respect to the map $Y\mapsto \lk_Y(\sigma)$. 

Denote by $\Delta'\subseteq \Delta_n$ the simplex spanned by the 
complementary vertices to vertices of $\sigma$. 
The link $L=\lk_Y(\sigma)$ is a random simplicial subcomplex of $\Delta'$. 
\begin{theorem}\label{linkslower}
The measure $\underline\lambda$ is the lower probability measure on the subcomplexes of $\Delta'$
with parameters
\begin{eqnarray}\label{linklower}
p'_\tau = p_\tau\cdot \prod_{\nu\subseteq \sigma}p_{\nu\tau}, \quad \tau\in \Delta'.
\end{eqnarray}
In the product $\nu$ runs over all  faces of $\sigma$.
\end{theorem}
\begin{proof}
We wish to compute probability that 
the link $L$ contains a given subcomplex $A\subseteq \Delta'$, i.e. 
$$\underline\lambda(A\subseteq L) = 
\sum_{A\subseteq L} \underline\lambda(A).$$
Using Corollary \ref{chr}(1), we find
\begin{eqnarray*}\underline\lambda(A\subseteq L) &=& \up(\sigma\in Y)^{-1}\cdot \up(\sigma\ast A\subseteq Y) \\ &=&
\left(\prod_{\nu\subseteq \sigma}p_\nu\right)^{-1}\cdot \left(\prod_{\nu\subseteq \sigma}p_\nu\cdot \prod_{\tau\subseteq A}p_\tau\cdot 
\prod_{\nu\subseteq \sigma, \tau\subseteq A} p_{\nu\tau}\right)\\
&=& \prod_{\tau\subseteq A} \left[p_\tau \cdot \prod_{\nu\subseteq \sigma} p_{\nu\tau}\right] = \prod_{\tau\subseteq A}p_\tau'.
\end{eqnarray*}
Our statement now follows from the intrinsic characterisation of the lower measure, see Corollary \ref{chr}(1).
\end{proof}
\begin{example}{\rm 
Consider the special case when the probability parameters $p_\tau=p_i$ depends only in the dimension $i=\dim \tau$. Since 
$$\dim (\nu\tau) = \dim \nu + \dim \tau +1,$$ 
and there are $\binom {k+1}{j+1}$ simplexes $\nu\subseteq \sigma$ of dimension $j=\dim \nu$, where $k=\dim \sigma$,
we see that
formula (\ref{linklower}) can be rewritten 
as follows
\begin{eqnarray}\label{linklower1}
p'_i = p_i \cdot \prod_{j=0}^{k}p_{i+j+1}^{\binom {k+1} {j+1}},
\end{eqnarray}
 This is consistent with Lemma 3.2 from \cite{farber2}. 
 }
 \end{example}

\section{Intersections and unions of Random Complexes}

\begin{lemma}
Consider the union $Y\cup Y'$ of two independent random simplicial complexes $Y, Y'\subseteq \Delta_n$ 
where $Y$ is sampled according 
to the upper probability measure $\overline {\Bbb P}$ with respect to a set of probability parameters $q_\sigma$ and $Y'$ is sampled 
according 
to the upper probability measure $\overline {\Bbb P}$ with respect to a set of probability parameters $q'_\sigma$. Then the union 
$Y\cup Y'\subseteq \Delta_n$ is a random simplical complex which is described by the upper probability measure with respect to the set of probability parameters $q_\sigma \cdot q'_\sigma$. In other words, the union $Y\cup Y'$ is an upper random simplicial complex with the set of probability parameters 
$$\sigma \mapsto p_\sigma +p'_\sigma - p_\sigma\cdot p'_\sigma.$$
where $p_\sigma =1-q_\sigma$ and $p'_\sigma =1-q'_\sigma$. 
\end{lemma}

\begin{proof} Let $B\subseteq \Delta_n$
be a simplicial complex. 
Clearly $ Y\cup Y'\subseteq B$ is equivalent to $Y\subseteq B$ and $Y'\subseteq B$. 
Since $Y$ and $Y'$ are independent, the probability that the union $Y\cup Y'$ is contained in $B$ equals the product 
\begin{eqnarray}
\overline {\Bbb P}(Y\subseteq B) \cdot \overline {\Bbb P'}(Y'\subseteq B) \, =\, \prod_{\sigma\not\in B}q_\sigma \cdot \prod_{\sigma\not\in B}q'_i
\, =\, \prod_{\sigma\not\in B}\left(q_\sigma \cdot q'_\sigma\right).
\end{eqnarray}
Our statement now follows from Corollary \ref{chr}(2). 
\end{proof}

The following Lemma generalises Lemma 4.1 from \cite{farber2}.
\begin{lemma}
Consider two sets of probability parameters $p_\sigma, p'_\sigma\in [0,1]$ associated to each simplex $\sigma\subseteq \Delta_n$.  
Let $\underline {\Bbb P}$ and $\underline {\Bbb P'}$ denote the lower probability measures determined by the probability parameters $p_\sigma$ and $p'_\sigma$. Suppose that $Y, Y'\subseteq \Delta_n$ are two independent random simplicial complexes where $Y$ is described according to the probability $\underline {\Bbb P}$ and $Y'$
is sampled according to $\underline {\Bbb P}'$. 
Then the intersection $Y\cap Y'\subseteq \Delta_n$ is a random simplical complex which is described by the lower probability measure with respect to the set of probability parameters $p_\sigma \cdot p'_\sigma$. 
\end{lemma}

%

\begin{proof} Let $A\subseteq \Delta_n$
be a simplicial complex. 
Clearly $A\subseteq Y\cap Y'$ is equivalent to $A\subseteq Y$ and $A\subseteq Y'$. 
Since $Y$ and $Y'$ are independent we see that the probability that the intersection $Y\cap Y'$ contains $A$ equals the product 
\begin{eqnarray}
\underline {\Bbb P}(A\subseteq Y) \cdot \underline {\Bbb P'}(A\subseteq Y') = \prod_{\sigma\in A}p_\sigma \cdot \prod_{\sigma\in A}p'_\sigma
=\prod_{\sigma\in A}\left(p_\sigma \cdot p'_\sigma\right).
\end{eqnarray}
Our statement now follows from Corollary \ref{chr}(1). 
\end{proof}

\section{Random pure complexes}\label{sec8}\label{sec9}

In this section we consider an interesting example of a random simplicial complex; the results of this section will be used in the proof of Theorem \ref{thmbetti2}. 

We fix a positive integer $k> 0$ and consider an upper random simplicial complex with probability parameters 
$$p_\sigma= 
\left\{
\begin{array}{ll}
p, & \mbox{for} \dim \sigma =k, \\ \\

0, &\mbox{otherwise}
\end{array}
\right.
$$
Here $p\in (0, 1)$ is a positive parameter, which typically depends on $n$. A random complex in this model is built by randomly selecting $k$-dimensional simplexes $\sigma\in \Delta_n$, chosen independently of each other, with probability $p$, and adding all faces of the selected simplexes. 

We ask under which conditions on the probability parameter $p$ the random pure $k$-dimensional complex contains the full $l$-dimensional skeleton 
$\Delta_n^{(\ell)}$, where $0\le \ell<k$? 

\begin{lemma}\label{lm801} (1) Suppose that 
\begin{eqnarray}\label{plus1}
p= \frac{(\ell+1)\log n +\omega}{\binom {n-\ell} {k-\ell}}
\end{eqnarray}
for a sequence $\omega\to \infty$. 
Then a random pure $k$-dimensional simplicial complex $Y$ contains the $\ell$-dimensional skeleton $\Delta_n^{(\ell)}$, a.a.s.  More precisely, under assumption (\ref{plus1}), a random complex $Y$ contains the $\ell$-skeleton $\Delta_n^{(\ell)}$ with probability at least 
$1-e^{-\omega}$. 

(2) If however 
\begin{eqnarray}\label{minus1}
p= \frac{(\ell+1)\log n -\omega}{\binom {n-\ell} {k-\ell}},
\end{eqnarray}
then $Y$ does not contain the $\ell$-skeleton $\Delta_n^{(\ell)}$, a.a.s. 
\end{lemma}
\begin{proof} For $\sigma\in \Delta_n$,  $\dim\sigma=\ell$, let $X_\sigma$ be a random variable which equals 1 if 
$\sigma\not\in Y$ and $0$ if $\sigma\in Y$. 
Then $X=\sum X_\sigma$ is the random variable counting the number of $\ell$-simplexes not in $Y$. 
We have $\E(X_\sigma) =q^{\binom {n-\ell} {k-\ell}}$ (where, as usual, $q=1-p$) and $$\E(X)=\binom {n+1}{\ell+1} \cdot q^{\binom {n-\ell} {k-\ell}}.$$
We show that the assumption (\ref{plus1}) implies that 
$\E(X)\to 0$. 
Indeed, 
$$\E(X) \le \binom{n+1}{\ell +1} e^{-p{\binom {n-\ell} {k-\ell}} } \le \exp\left((\ell+1)\log n - p \binom {n-\ell}{k-\ell}\right)
 = e^{-\omega}\to 0.$$
The statement (1) now follows from the Markov inequality.

To prove (2) we want to apply the inequality 
$${\Bbb P}(X>0) \ge \frac{({\mathbf E}X)^2}{{\mathbf E}(X^2)},$$
see (3.3) on page 54 of \cite{JLR}. 
We shall assume that $p$ satisfies (\ref{minus1}) and will show that $\frac{\Bbb E(X)^2}{\Bbb E(X^2)}\to 1$. We have $$X^2=\sum_{(\sigma, \tau)} X_\sigma X_\tau,$$ 
where $(\sigma, \tau)$ runs over all pairs of $\ell$-dimensional simplexes of $\Delta_n$, and 
\begin{eqnarray}\label{eq}
{\Bbb E}(X_\sigma X_\tau)= 
\left\{
\begin{array}{ll}
q^{2\binom {n-\ell}{k-\ell} -\binom {n-x}{k-x}},&\mbox{if} \quad x\le k,\\ \\
q^{2\binom {n-\ell}{k-\ell}}, &\mbox{if} \quad x>k.
\end{array}
\right.
\end{eqnarray}
where
$x=\dim (\sigma\cup\tau)$.
Both cases in this formula can be written as in the upper row since 
$\binom r s =0$ for $s<0$. 
To explain formula (\ref{eq}) we note that ${\Bbb E}(X_\sigma X_\tau)$ equals probability that neither of the simplexes $\sigma, \tau$ are included in $Y$. There are $\binom {n-\ell}{k-\ell}$  simplexes of dimension $k$ containing $\sigma$ and the same number of 
$k$-simplexes contain $\tau$. However in this count we include the $k$-simplexes containing both $\sigma$ and $\tau$ twice, and this fact is reflected in the term $\binom {n-x}{k-x}$. 

Denoting $$d=\dim(\sigma\cap \tau) = 2\ell -x$$ we obtain
\begin{eqnarray*}
{\Bbb E}(X^2) &=& \sum_{d=-1}^{\ell} \binom {n+1}{\ell+1} \cdot  \binom{\ell+1}{d+1}\cdot  \binom{n-\ell}{\ell -d} \cdot q^{2{\binom {n-\ell}{k-\ell}}-\binom{n-x}{k-x}}
 \end{eqnarray*}
 and hence
 \begin{eqnarray}\label{fraction}
 \frac{\Bbb E(X^2)}{\Bbb E(X)^2} = \sum_{d=-1}^{\ell} \frac{\binom{\ell+1}{d+1}\binom {n-\ell}{\ell - d}}{\binom{n+1}{\ell+1}}\cdot
 q^{-\binom{n-x}{k-x}}.
 \end{eqnarray}
 Here $x=2\ell -d$. 
 
The term of the sum (\ref{fraction}) with  $d=\ell$ and $x=\ell$ is 
$$
{\binom {n+1}{\ell+1}}^{-1} q^{-\binom{n-\ell}{k-\ell}}= \Bbb E(X)^{-1}.
$$
We show below that assumption (\ref{minus1}) implies that $\Bbb E(X)\to \infty$ and hence this term tends to $0$. 
There exists $C>0$ and $N>0$ such that for any $n>N$ one has $\binom {n+1}{\ell +1} >C n^{\ell+1}$ 
and $\log(1-p) >-p(1+p)$ for sufficiently small $p>0$. 
Hence 
\begin{eqnarray*}\log \Bbb E(X) &>& (\ell+1)\log n + \binom {n-\ell}{k-\ell} \log(1-p)+ C'\\
&>& (\ell +1) \log n - \binom {n-\ell}{k-\ell} p(1+p)+C' \\ &=& \omega(1+p)- (\ell+1) \cdot p\cdot \log n+C'
\end{eqnarray*}
It is easy to see that our assumption (\ref{minus1}) and also $\ell <k$ imply that $p\log n \to 0$. Hence, we see that the summand of (\ref{fraction}) with $d=\ell$
tends to zero.  

Consider now the term  of (\ref{fraction}) with $d=-1$ and $x=2\ell+1$; it equals 
$$\frac{\binom{n-\ell}{\ell+1}}{\binom {n+1}{\ell+1}} q^{-\binom{n-2\ell-1}{k-2\ell-1}}.$$
We show below that this term tends to $1$ as $n\to\infty$. 
For $k< 2\ell+1$ our claim is obvious since the coefficient $\frac{\binom{n-\ell}{\ell+1}}{\binom {n+1}{\ell+1}}$ tends to $1$. 
In the sequel we shall assume that $k\ge 2\ell+1$. We observe that (\ref{minus1}) implies that 
$$p\binom{n-2\ell - 1}{k-2\ell-1} \sim p n^{k-2\ell -1}\to 0$$
and therefore (using Remark \ref{rk1014}) we obtain
$$q^{\binom{n-2\ell-1}{k-2\ell -1}}= 1 - p\binom{n-2\ell-1}{k-2\ell -1}+O\left(p^2\binom{n-2\ell-1}{k-2\ell -1}^2\right)$$
which converges to $1$.

It remains to show that any summand of (\ref{fraction}) with $-1<d<\ell$ tends to zero. If the symbol $S_d$ represents this summand, then 
$$S_d^{-1} = \frac{\binom{n+1}{\ell+1}}{\binom{\ell+1}{d+1}\binom{n-\ell}{\ell-d}}\cdot q^{\binom{n-x}{k-x}}$$
and we show that $S_d^{-1}\to \infty$. Using the inequalities $\binom {n+1}{\ell+1} > Cn^{\ell+1}$ and $\binom{n-\ell}{\ell-d}< n^{\ell-d}$ we obtain
\begin{eqnarray*}
\log(S_d^{-1}) &>& (\ell+1) \log n - (\ell-d)\log n + \binom{n-x}{k-x}\log(1-p) +C'\\
&>& (d+1)\log n - 2\binom{n-x}{k-x} p +C'
\end{eqnarray*}
Above we have used the inequality $\log(1-p) >-2p$ which is valid for sufficiently small $p>0$. Since $d\ge 0$ we have $(d+1)\log n\to \infty$. On the other hand, since $x>\ell$ we have $\binom{n-x}{k-x} p\sim pn^{k-x}\to 0$. This completes the proof. \end{proof}

\begin{remark}\label{rk82}
{\rm 
Equation (\ref{plus1}) can equivalently be written as 
\begin{eqnarray}\label{plus111}
p=\frac{(\ell+1)\cdot (k-\ell)!\cdot \log n +\omega}{n^{k-\ell}};
\end{eqnarray}
similarly, equation (\ref{minus1}) can be written as 
\begin{eqnarray}\label{minus111}
p=\frac{(\ell+1)\cdot (k-\ell)!\cdot \log n -\omega}{n^{k-\ell}},
\end{eqnarray}
where $\omega\to \infty$. 
}
\end{remark}
%
%
%
%
%
%
%
%

{\bf A special class of random pure complexes.}

In the rest of this section we consider a special class of random pure complexes; we shall give a full description of their Betti numbers. 

For $k> 0$, consider an upper random $k$-dimensional  pure complex $Y$
with the probability parameter of the form 
\begin{eqnarray}\label{011}
p=n^{-\alpha},\end{eqnarray} where 
\begin{eqnarray}\label{01}
\alpha\in (0,1).\end{eqnarray}

\begin{prop} \label{cor901} Let $Y$ be a random $k$-dimensional pure complex with respect to the upper measure with probability parameter $p$ of the form (\ref{011}) and with exponent $\alpha$ satisfying 
(\ref{01}). Then: (a) $Y$ contains the full $(k-1)$-dimensional skeleton, $\Delta_n^{(k-1)}$, a.a.s.; (b) 
the $k$-dimensional face number $f_k(Y)$ satisfies
\begin{eqnarray}\label{between1}
(1-t_n) \cdot \frac{n^{k+1-\alpha}}{(k+1)!}  \, \le\,  f_k(Y) \, \le\,  (1+t_n) \cdot \frac{n^{k+1-\alpha}}{(k+1)!},
\end{eqnarray}
a.a.s., where $t_n\in [0,\infty)$ is a sequence tending to zero, $t_n\to 0$. 
\end{prop}

\begin{proof}
(a) follows from Lemma \ref{lm801} and Remark \ref{rk82}.

To prove (b) we note that the number $f_k(Y)$ of $k$-dimensional faces of $Y$ is a binomial random variable, 
${\rm Bi}(\binom {n+1}{k+1}, n^{-\alpha})$, and hence 
\begin{eqnarray}\label{203}
\E(f_k) =\binom {n+1}{k+1}\cdot n^{-\alpha} \, =\,  \frac{n^{k+1-\alpha}}{(k+1)!}\cdot (1+o(1)).
\end{eqnarray}

We may use the Chernoff bound, see Theorem 2.1 of \cite{JLR}, which states that for any $\tau\ge 0$, 
$$\Bbb P(f_k\ge \E(f_k) +\tau) \le \exp\left(-\frac{\tau^2}{2(\E(f_k) +\tau/3)}\right), $$ and 
$$\Bbb P(f_k\le \E(f_k) -\tau) \le \exp\left(-\frac{\tau^2}{2\E(f_k)}\right). $$
We shall apply these bounds with $\tau = \Bbb E(f_k)^{2/3} =t_n\Bbb E(f_k)$ 
where $t_n = \Bbb E(f_k)^{-1/3} =o(1)$. 
We obtain
$$\Bbb P(f_k \geq (1+t_n)\cdot\Bbb E(f_k))\leq \exp\left(-\frac{\tau^2}{2(\Bbb E(f_k) +\tau/3)} \right)\leq \exp\left(-\frac{\Bbb E(f_k)^{1/3}}{4}\right)$$
and
$$\Bbb P(f_k \leq (1-t_n)\cdot\Bbb E(f_k))\leq \exp\left(-\frac{\Bbb E(f_k)^{1/3}}{2}\right).$$
Since $\Bbb E(f_k)^{1/3} \to \infty$ we see
$
(1-t_n) \cdot \Bbb E(f_k) \, \le\,  f_k \, \le\,  (1+t_n) \cdot \Bbb E(f_k),
$
a.a.s. 
\end{proof}

\begin{corollary}\label{cor92} Let $Y$ be a random $k$-dimensional pure complex with respect to the upper measure with probability parameter $p$ of the form (\ref{011}) and with exponent $\alpha$ satisfying 
(\ref{01}). The $k$-dimensional Betti number $b_k(Y)$ satisfies:
\begin{eqnarray}\label{bky1}
(1- t'_n)\cdot \frac{n^{k+1-\alpha}}{(k+1)!} \le b_k(Y) \le (1+ t'_n)\cdot \frac{n^{k+1-\alpha}}{(k+1)!},
\end{eqnarray}
a.a.s., for a sequence $t'_n\to 0$.
\end{corollary}
\begin{proof}
We apply the Morse inequalities 
$$f_k(Y) -f_{k-1}(Y) \le b_k(Y)\le f_k(Y).$$  
By Proposition \ref{cor901}, $f_k(Y) \sim C\cdot n^{k+1-\alpha},$ and $f_{k-1}(Y) \sim C'\cdot n^k$, a.a.s. This obviously implies (\ref{bky1}) since $\alpha<1$. 
\end{proof}

Next we consider the Betti numbers of $Y$ below dimension $k$:
\begin{corollary} \label{lm93}
Let $Y$ be a random $k$-dimensional pure complex with respect to the upper measure with probability parameter $p$ of the form (\ref{011}) and with exponent $\alpha$ satisfying 
(\ref{01}). The reduced Betti numbers of $Y$ below dimension $k$ vanish, i.e. $\tilde b_\ell(Y)=0$ for all $\ell<k$, a.a.s.
\end{corollary}
\begin{proof} 
We know that $Y$ contains the full $(k-1)$-dimensional skeleton $\Delta_n^{(k-1)}$ a.a.s. 
The latter space is acyclic in all dimensions $<k-1$. 
To prove Corollary \ref{lm93} we only need to show that the Betti number $b_{k-1}(Y)=0$ vanishes, a.a.s.
Now, we can view $Y$ as a Linial - Meshulam random simplicial complex 
with probability parameter $p=n^{-\alpha}$ where $\alpha<1$. It is well known (see \cite{walmesh})
that in this situation the rational homology in dimension $k -1$ vanishes, i.e. $b_{k-1}(Y)=0$, a.a.s. 
As a clarification we note that, while \cite{walmesh} operates with cohomology groups with finite coefficients, vanishing of cohomology groups with coefficients in $\Z_2$ implies vanishing of cohomology groups with rational 
coefficients, i.e. vanishing of the Betti numbers. 
\end{proof}

%

\section{The Alexander Duality}\label{ad}
In this section we continue the study of \S \ref{du} by showing that random simplicial complexes in lower and upper model are dual to each in the sense of Spanier - Whitehead duality. 
This implies that homology and cohomology of the lower and upper complexes satisfy the Alexander duality relation. 

We start this section by introducing basic notions of duality; this material is well known and is included for convenience of the reader. 

Recall that $[n]$ denotes the set $\{0, 1, \dots, n\}$ and $\Delta_n$ denotes the $n$-dimensional simplex spanned by $[n]$. The symbol $\partial \Delta_n$ denotes the boundary of the simplex, i.e. the union of all proper faces $\sigma\subseteq \Delta_n$. 
Topologically $\partial \Delta_n$ is a sphere of dimension $n-1$. 

\subsection{The dual simplicial complex}\label{xprime}
In this section we describe a combinatorial duality construction for simplicial complexes.  More precisely, for a simplicial subcomplex 
$X\subseteq \pa \Delta_n$ we construct a simplicial complex $X'\subseteq \Delta_n$ which is homotopy equivalent to the complement $\partial \Delta_n-X$ of $X$ in the sphere $\partial\Delta_n$. 


Let $X\subseteq\pa\Delta_n$ be a simplicial subcomplex. 
Define {\it the dual complex} $X'$ as an abstract simplicial complex with the vertex set $E(X)$ (the set of all external faces of $X$) and
a set of external faces $\sigma_1,\dots,\sigma_k\in E(X)$ of $X$ forms a $(k-1)$-simplex of $X'$ if the union of their vertex sets is a proper 
subset of $[n]$, i.e. $$\cup_{i=1}^k V(\sigma_i) \neq [n].$$

\begin{prop}\label{prop:X'homotopy}
The geometric realisation of the simplicial complex $X'$ is homotopy equivalent to the complement $\partial\Delta_n - X$. \end{prop}
\begin{proof}
For any $\sigma\in E(X)$ let $\st(\sigma)$ denote  $\st(\sigma) =\st_{\partial\Delta_n}(\sigma)$ -- the star of $\sigma$ viewed as a subcomplex of $\partial \Delta_n$.
Recall that $\st(\sigma) =\st_{\partial\Delta_n}(\sigma)$ is defined as the union of all open simplexes $\tau\subseteq \partial \Delta_n$ whose closure contains 
$\sigma$. 

The family of stars
$\mathcal{U} = \{\St(\sigma)\}_{\sigma\in E(X)}$
forms a contractible open cover of the complement $\partial\Delta_n - X$. 
Indeed,  for $\sigma\in \ext(X)$ we obviously have $\st(\sigma) \cap X =\varnothing$ which gives the inclusion 
$$\bigcup_{\sigma\in \ext(X)}\st(\sigma)\subseteq \partial \Delta_n-X.$$
This is in fact an equality, i.e. for any open simplex $\tau\subseteq \partial\Delta_n$ with $\tau\not\subseteq X$ there is 
a face $\sigma\in \ext(X)$ such that $\tau\subseteq \st(\sigma)$. 
Indeed, given $\tau\not\in X$ let $\sigma\su \tau$ be a minimal face of $\tau$ not in $X$. Then $\sigma\in E(X)$ and $\tau\su \st(\sigma)$. 


Note that the cover $\mathcal U$ has the property that each intersection
$$\St(\sigma_1)\cap\dots\cap \St(\sigma_k) = \St(\sigma)$$
is a star of a simplex $\sigma$, 
where $\sigma$ has the vertex set $V(\sigma) = \cup_{i=1}^k V(\sigma_k)$. Thus, every such intersection is either contractible or empty, and it is empty precisely when $\cup_{i=1}^k V(\sigma_k) = [n]$. 
The result now follows by noting that the nerve of $\mathcal{U}$ is exactly the simplcial complex $X'$ and then applying the Nerve Theorem, see \cite{hatcher}, Corollary 4G.3.
\end{proof}

\begin{example}{\rm 
For $n=2$ let $X\subseteq \pa\Delta_2$ be the vertex set, i.e. $X=\{0, 1, 2\}$. It is a 0-dimensional subcomplex 
whose complement $\pa\Delta_2-X$ is a circle with 3 punctures; it has 3 connected components, each is contaractible.
The set of external simplexes 
$E(X)$ consists of all edges, $$E(X)=\{(ij); i<j, \quad i, j\in[2]\}, \quad |E(X)|=3.$$ The dual complex $X'$ has no edges, i.e. $X'$ is a 3 point set. 

%
%

}
\end{example}

Applying the Alexander duality theorem combined with Proposition \ref{prop:X'homotopy} we obtain:

\begin{prop} For any proper simplicial subcomplex $X\subseteq \pa\Delta_n$ and for any abelian group $G$ one has $$H^j(X';G) \simeq 
H_{n-2-j}(X;G), \quad \mbox{where}\quad  j=0, 1, \dots, n-2.$$
\end{prop}

\subsection{The dual complex $c(X)$ of Bj\"orner and Tancer \cite{BT} }\label{sec:alexander}
Recall that in \S \ref{du} we defined  its \textit{combinatorial Alexander dual} $c(X)$ for any simplicial subcomplex $X\subseteq\pa\Delta_n$.
The maximal simplices of $c(X)$ are in bijective correspondence with the external faces of $X$, $\ext(X)$. More precisely, we have the following:

\begin{lemma} Let $Y\subseteq \pa\Delta_n$ be a simplicial subcomplex. Then
\begin{eqnarray}\label{32}
f_i(Y) +  f_{n-1-i}(c(Y)) \, =\,  \binom {n+1} {i+1}, \quad i=0, 1, \dots, n-1.
\end{eqnarray}
Here $f_i(Y)$ denotes the number of $i$-dimensional simplexes in $Y$. 
A simplex $\sigma\subseteq \Delta_n$ is an external simplex for $Y$ if and only if 
the dual simplex $\widehat{\sigma}$ is a maximal simplex of the complex $c(Y)$. 
In particular we have
\begin{eqnarray}\label{33}
e_i(Y) = d_{n-i-1}(c(Y)), \quad i=0, \dots, n-1,\end{eqnarray}
where $e_i(Y)$ denotes the number of external $i$-dimensional faces of $Y$ and $d_j(Y)$ denotes the number of $j$-dimensional maximal simplexes of $Y$. 
\end{lemma}
\begin{proof}
The map $\sigma\mapsto \hat\sigma$ is a bijection between the set of $i$-dimensional non-simplexes of $Y$ and the set of $(n-i-1)$-dimensional simplexes of the dual $c(Y)$; this proves (\ref{32}). By Lemma \ref{me} this map is a bijection between the set 
$E_i(Y)$ the set of $i$-dimensional external simplexes of $Y$ and the set of maximal simplexes of $c(Y)$ of dimension $n-i-1$; this proves (\ref{33}). 
%
%
\end{proof}

\begin{lemma}\label{lem:dualiso}
The nerve of the cover of $c(X)$ by its maximal simplices is isomorphic to the simplicial complex $X'$ (as defined in \S \ref{xprime}). 
\end{lemma}
\begin{proof}
Let $\mathcal{M}$ denote the cover of $c(X)$ by maximal simplices. Consider a set of maximal simplexes $\{\sigma_1, \dots, 
\sigma_k\}$, where $\sigma_i\in \mathcal M$. Each dual simplex $\hat\sigma_i$ is external for $X$. 
The intersection $\cap_{i=1}^k \sigma_i$ is a simplex with the vertex set $\cap_{i=1}^k V(\sigma_i)$ and the intersection 
$\cap_{i=1}^k \sigma_i$ is non-empty if and only if $\cap_{i=1}^k V(\sigma_i)\not=\varnothing$. We see that any nonempty intersection is contractible. Since $V(\hat\sigma_i)$ is the complement of $V(\sigma_i)$, we obtain that $\cap_{i=1}^k V(\sigma_i)=\varnothing$ if and only if 
$\cup_{i=1}^k V(\hat\sigma_i) =[n]$. Therefore, we see that the nerve of $\mathcal M$ can be described as the simplicial complex with the vertex set $E(X)$ where a set of external simplexes forms a simplex if and only if the union of their vertex sets is not equal $[n]$. This complex coincides with $X'$ as defined in \S \ref{xprime}. 
\end{proof}


\begin{corollary}\label{cor:dualshomoequ}
For a simplicial subcomplex $X\subseteq \partial \Delta_n$, the geometric realisation of the simplicial complex $c(X)$ is homotopy equivalent to $ X'$ and to the complement $\pa\Delta_n-X$. 
\end{corollary}
\begin{proof} The cover $\mathcal M$ by maximal simplexes satisfies the conditions of the Nerve Theorem, see \cite{hatcher}, Corollary 4G.3. The first claim follows from the previous Lemma. The second claim follows from Proposition \ref{prop:X'homotopy}. 
\end{proof}
\begin{corollary}\label{cor97} For any simplicial subcomplex $X\subseteq \pa \Delta_n$ and for any abelian group $G$ one has $$H^j(c(X);G) \simeq 
H_{n-2-j}(X;G), \quad \mbox{where}\quad  j=0, 1, \dots, n-2.$$
\end{corollary}

Taking here $G=\Bbb Q$ we obtain equality for the Betti numbers: $$b_j(c(X))= b_{n-2-j}(X), \quad j=0, 1, \dots, n-2.$$

Next we restate Proposition \ref{pn} as follows: 

\begin{prop} For a fixed $n$ consider two probability spaces $(\Omega_n^\ast, \op)$ and $(\Omega_n^\ast, \up')$ where 
the probability measure $\op$ is defined with respect to a set of probability parameters $p_\sigma$ and the probability measure $\up'$ is 
defined with respect to a set of probability parameters $p'_\sigma$ satisfying 
$$p'_\sigma = q_{\hat \sigma}= 1 - p_{\hat \sigma}.$$
The map $c: (\Omega_n^\ast, \op)\to (\Omega_n^\ast, \up')$, where $X\mapsto c(X)$, is an isomorphism  of probability spaces. 
For an integer $j\in [n]$, consider the $j$-dimensional Betti number $$b_j: \Omega^\ast \to \Z$$ and its 
distribution functions $F_j^{\op}(x)$ and $F_j^{\up'}(x)$ with respect to the measures $\op$ and $\up'$ correspondingly. Then 
\begin{eqnarray}
F_j^{\op}(x) \, \equiv\,  F_{n-2-j}^{\up'}(x).
\end{eqnarray}
\end{prop}
\begin{proof}
This follows by combining Corollary \ref{cor97} and Propositon \ref{pn}. 
\end{proof}
Note that the distribution function $F_j^{\op}$ is defined by the equality
$$F_j^{\op}(x) =\op(b_j(Y)\le x)$$
and similarly, 
$$F_j^{\up'}(x) =\up'(b_j(Y)\le x).$$

Summarising, we see that for a fixed $n$, studying Betti numbers in the upper model reduces to studying Betti numbers in the lower model and vice versa. However, in the limit when $n\to \infty$ one needs to deal with the dimension shift $i\to n-1-i$ which creates an additional technical difficulty.

\section{The notion of critical dimension in the lower model}

In a recent paper \cite{farber4} the authors studied Betti numbers of random simplicial complexes $Y$  in the lower model. It was shown   that the 
Betti numbers have a very specific pattern which can be described using the notion of {\it a critical dimension} $k_\ast$. 
Roughly, it was established in \cite{farber4} that homologicaly a random simplicial complex in the lower model can be well approximated by a wedge of spheres of dimension $k_\ast$. 

More precisely (see \cite{farber4} for more detail), the critical dimension $k_\ast$ in the lower model satisfies: 
\begin{enumerate}
\item The Betti number $b_{k_\ast}(Y)$ in the critical dimension is large,  $$b_{k_\ast}(Y) \sim C \cdot n^{a_{k_\ast}},$$ where $a_{k_\ast}>0$, $C>0$ are constants, a.a.s.; 
\item The reduced Betti numbers $\tilde b_j(Y)$ in all dimensions below the critical dimension $j<k_\ast$ vanish, a.a.s.; 
\item The Betti numbers $b_j(Y)$ in dimensions above the critical dimension $j>k_\ast$ are \lq\lq significantly smaller\rq\rq\,  than $b_{k_\ast}(Y)$, a.a.s. One possibility to clarify the words \lq\lq significantly smaller\rq\rq\, is by means of an upper bound 
$b_j(Y) \le n^{a_j}$, a.a.s., where $a_j<a_{k\ast}$. 
\item Homology groups in dimensions above the critical dimension are generated by primitive cycles of bounded size (cf. Theorems 20, 21 in \cite{farber4});
\item If the critical dimension is positive $k_\ast>0$ then the random complex $Y$ is connected, a.a.s.
\item If the critical dimension is greater than 1, $k_\ast>1$, then the fundamental group $\pi_1(Y)$ has property (T); 
\item If the critical dimension is greater than 2, $k_\ast> 2$, then $Y$ is simply connected, a.a.s.
\item The critical dimension $k_\ast$ and the exponents $a_{k_\ast}$ can be explicitly calculated through the probability parameters $p_\sigma$. 
\end{enumerate}

In the following section we shall investigate the Betti numbers in the upper model hoping to find a similar pattern of behaviour of the Betti numbers. However, as we shall see, the Betti numbers in the upper model behave differently. 

For convenience of the reader we include below the definition of the critical dimension $k_\ast$ in the lower model under the assumption
that $p_\sigma=0$ for $\dim \sigma >r$ and 
$$p_\sigma = n^{-\alpha_i}, \quad \mbox{for}\quad i=0, 1, \dots, r.$$
In other words, we are considering random simplicial complexes of dimension less or equal than $r$. 
%
%

Define linear maps $\psi_k: \Bbb R^{r+1}\to \Bbb R$ by
\begin{align*}
\psi_k(\alpha)=\sum_{i=0}^r \binom{k}{i}\alpha_i,\quad\quad k=0,\ldots, r.
\end{align*}

Since $\binom k i < \binom {k+1} i$ for $i>0$ we see that
\begin{align*}\label{ineq1}
\psi_0(\alpha)\leq\psi_1(\alpha)\leq\psi_2(\alpha)\leq\ldots\leq\psi_r(\alpha).
\end{align*}
Moreover, if for some  $j\geq 0$ one has $\psi_j(\alpha)<\psi_{j+1}(\alpha)$ then
\begin{equation*}\label{ineq2}\psi_j(\alpha)<\psi_{j+1}(\alpha)<\ldots<\psi_r(\alpha).
\end{equation*}
We introduce the following convex domains in $\Bbb R^{r+1}_+$:
\begin{eqnarray}\label{defdk}
\D_k = \{\alpha\in \R^{r+1}_+\, :\,  \psi_k(\alpha)<1<\psi_{k+1}(\alpha)\},
\end{eqnarray}
where $k=0, 1, \dots, r-1$. 
One may also introduce the domains
$$\D_{-1}=  \{\alpha\in \R^{r+1}_+\, :\,  1<\psi_0(\alpha)\}, \quad \D_{r}=  \{\alpha\in \R^{r+1}_+\, :\,  \psi_r(\alpha)<1\}.$$
The domains $$\D_{-1}, \D_0, \D_1,\dots, \D_r$$ are disjoint and their union is $$\bigcup_{j=-1}^r \D_j \, = \, \R^{r+1}_+- \bigcup_{i=0}^r H_i,$$
where $H_i$ denotes the hyperplane 
$$H_i=\{\alpha\in \R^{r+1}; \psi_i(\alpha)=1\}.$$ 
%

\begin{defn}\label{critical} The critical dimension $k=k_\ast$ of a random simplicial complex $Y$ in the lower model is defined by 
the condition
\begin{eqnarray}\label{crit}
\alpha\in \D_k,\quad \mbox{where}\quad k=-1, 0, \dots, r,\end{eqnarray}
where $\alpha=(\alpha_0, \alpha_1, \dots, \alpha_r)$ is the vector of the exponents. 
\end{defn}

In other words, the critical dimension $k=k_\ast$ satisfies the inequalities
\begin{eqnarray}\label{crit5}\sum_{i=0}^r \binom k i \alpha_i \, <\, 1\, <\,  \sum_{i=0}^r \binom {k+1} r \alpha_i\end{eqnarray}

We must emphasise that the notion of critical dimension in the lower model is defined only for {\it generic} vectors of exponents $\alpha=(\alpha_0, \dots, \alpha_r)$, i.e. when equalities do not happen in (\ref{crit5}). 

\section{Critical dimension and spread in the upper model}

In this section we introduce the notions of {\it a critical dimension} and of {\it a spread} for the upper model and explore its relevance to the properties of face numbers of random simplicial complexes. 

\subsection{}\label{101} As above, let $\Omega^\ast_n$ denote the set of simplicial subcomplexes of $\Delta_n$. 

In this section we shall consider the upper probability measure $\op$ on $\Omega^\ast_n$ under the following 
 assumptions on the probability parameters $p_\sigma$: 
 
 {\it 

(a) all probability parameters $p_\sigma=0$ vanish for $\dim \sigma>r$ where $r\ge 0$ is a fixed integer. 

(b) For $i\le r$ one has $p_\sigma=n^{-\alpha_i}$ where $i=\dim \sigma$ and $\alpha_i>0$ is a fixed positive real number. 

(c) The exponents $\alpha_i$ are not integers, $\alpha_i\notin \Z$, where $i=0, 1,  \dots, r$. 

(d) All the differences $\alpha_i-\alpha_j\notin \Z$ are not integers, where $i\not=j$, $i, j =0, 1,  \dots, r$.}

We note that (b) in particular requires that $p_\sigma$ depends only on the dimension of simplex $\sigma$.
The assumptions (c) and (d) are satisfied for a "generic" set of exponents $\alpha_0, \dots, \alpha_r$. Many results stated below are valid with relaxed assumptions (a) - (d) but then the statements require more complicated notations and explanations. For this reason we decided to restrict ourselves to the assumptions (a) - (d) aiming at having the most transparent statements and definitions valid for a generic set of exponents.

By Remark \ref{rm11} we know that the measure $\op$ is supported on the set of all $r$-dimensional simplicial complexes $Y\subseteq \Delta_n$.

\subsection{} Next we introduce the following notations.
Denote $$\beta_i=i+1-\alpha_i$$ and 
\begin{eqnarray}\label{betastar}
\beta^\ast=\max\{\beta_0, \beta_1, \dots, \beta_r\}, \quad i=0, 1, \dots, r.\end{eqnarray}
 We set 
\begin{eqnarray}k^\ast=\lfloor \beta^\ast\rfloor.\end{eqnarray}
Note that $k^\ast=k^\ast(\alpha)$ is an integer depending on the initial vector of exponents $\alpha=(\alpha_0, \dots, \alpha_r)$. 
Besides, $$k^\ast<\beta^\ast<k^\ast+1,$$
the strong inequalities hold due to our genericity assumption (c). 

\begin{definition}\label{def102}
The integer $k^\ast=k^\ast(\alpha)$ will be called the critical dimension of the random simplicial complex $Y$ in the upper model. 
\end{definition}

Due to our assumption (d) there exists a single index $i^\ast\in \{0, \dots, r\}$ such that $\beta_{i^\ast}=\beta^\ast$.

The following observation will be useful:

\begin{lemma}
One has $k^\ast\le i^\ast.$
\end{lemma}
\begin{proof}
This follows from the inequalities
$$k^\ast=\lfloor \beta^\ast\rfloor<\beta^\ast=\beta_{i^\ast}<i^\ast+1.$$
\end{proof}

\begin{definition}\label{defspread}
We shall call the non-negative integer $i^\ast - k^\ast=s=s(\alpha)$ the spread.
\end{definition}

\begin{example}{\rm 
Let us show that the condition $k^\ast(\alpha)<0$ is equivalent to the property that $\op(\varnothing)=1$, a.a.s. Indeed, 
$$\op(\varnothing) = \prod_{\sigma\in \Delta_n} q_\sigma = \prod_{i=0}^r q_i^{\binom{n+1}{i+1}} = \prod_{i=0}^r (1-n^{-\alpha_i})^{\binom{n+1}{i+1}} .$$
One has $k^\ast<0$ if and only if  $\beta_i<0$ for any $i=0, 1, \dots, r$. 
Since $\binom {n+1}{i+1} n^{-\alpha_i} = n^{\beta_i} \cdot (((i+1)!)^{-1} +o(1))$ we may apply Remark \ref{rk1014} to obtain
$$\op(\varnothing) = 1- \sum_{i=0}^r \frac{n^{\beta_i}}{(i+1)!} +o(1) = 1+o(1).$$

On the other hand, suppose that $\op(\varnothing)\to 1$. Since 
$$\op(\varnothing)=\prod_{i=0}^r(1-n^{-\alpha_i})^{\binom{n+1}{i+1}}$$
we see (since each term in this product is smaller than 1) that for each $i=0, \dots, r$ one must have 
$(1-n^{-\alpha_i})^{\binom{n+1}{i+1}}\to 1$ which implies $\binom {n+1}{i+1} n^{-\alpha_i}\to 0$, i.e. $\beta_i<0$. 
\qed
}
\end{example}
\begin{remark} \label{rk1014} {\rm 
We have used the following fact: If $N\to \infty$ and $Nx\to 0$, $x>0$ then 
\begin{eqnarray}\label{41}
(1-x)^N = 1-Nx +(xN)^2(1/2 +o(1)).
\end{eqnarray}
To prove (\ref{41}) one observes that in the expansion
$$(1-x)^N = \sum_{k=0}^N (-1)^k \binom N k x^k$$ one has $\binom N k x^k > \binom N {k+1} x^{k+1}$ for $k\ge 2$ assuming that $xN<3$. 
Hence we may apply the known result about alternating series with decreasing terms. 
}
\end{remark}

%
\subsection{} For any $\ell=0, 1, \dots, r$ consider the function $$f_\ell: \Omega^\ast_n\to \Z$$ which assigns the number of $\ell$-dimensional faces $f_\ell(Y)$ to a random subcomplex $Y\subseteq \Delta_n$. 
Using Example \ref{ex3} we find
\begin{eqnarray}\label{expectationk}
\E(f_\ell) = \sum_{\dim\sigma =\ell} \left(1- \prod_{\tau\supseteq\sigma}q_\tau\right) = \binom {n+1}{\ell+1}\cdot \left(1 - 
\prod_{i=\ell}^r q_i^{\binom {n-\ell}{i-\ell}}\right),\end{eqnarray}
where $$q_i=1-n^{-\alpha_i}, \quad i=0, 1, \dots, r.$$

\begin{lemma} \label{lm101} Let $k^\ast$ denote the critical dimension as defined in Definition \ref{def102}. 
Then for any $\ell<k^\ast$ 
a random complex $Y$ contains the full $\ell$-dimensional skeleton of $\Delta_n$, a.a.s. 
More precisely, one has $$f_\ell(Y)=\binom {n+1} {\ell+1}$$ with probability at least 
\begin{eqnarray}\label{36}
1- n^{k^\ast} \cdot \exp\left(- \frac{n^{\{\beta^\ast\}}}{2\cdot (s(\alpha)+1)!}\right).
\end{eqnarray}
where $\{\beta^\ast\}>0$ denotes the fractional part of $\beta^\ast$, i.e. $\{\beta^\ast\}=\beta^\ast - k^\ast$ and $s(\alpha)\ge 0$ is the spread (see Definition \ref{defspread}).
\end{lemma}
\begin{proof} It is enough to take $\ell=k^\ast-1$. By the definition, $k^\ast=\lfloor \beta^\ast\rfloor$
and $\beta^\ast = \beta_{i^\ast}$, where $i^\ast= k^\ast +s(\alpha)$ with $s(\alpha)\ge 0$ being the spread. 
Denote $k=i^\ast$ and consider a pure random $k$-dimensional simplicial complex $Z$ with probability parameter $p=n^{-\alpha_k}$ as defined in \S \ref{sec8}. Clearly $Z$ can be viewed as a subcomplex of $Y$ and we may engage Lemma \ref{lm801} to help us to decide whether $Z$ contains the $(k^\ast-1)$-dimensional skeleton 
of the simplex $\Delta_n$. Writing 
\begin{eqnarray}\label{ppp}
p= n^{-\alpha_k} = \frac{k^\ast \cdot \log n +\omega}{\binom {n-k^\ast+1}{k-k^\ast +1}}\end{eqnarray}
and noting that $k-k^\ast = i^\ast - k^\ast =s(\alpha)$ and 
$$s(\alpha)+1 -\alpha_k = i^\ast +1 - \alpha_{i^\ast} - k^\ast = 
\beta^\ast - \lfloor \beta^\ast \rfloor = \{\beta^\ast\}$$
we can solve (\ref{ppp}) for $\omega$ and obtain
$$\omega \, \ge\,  \frac{n^{\{\beta^\ast\}}}{2\cdot (s(\alpha)+1)!} - k^\ast \cdot \log n. $$
Now, by Lemma \ref{lm801} we obtain that $Y$ contains the full $(k^\ast-1)$-dimensional skeleton with probability at least 
$$1- e^{-\omega}\ge 1- n^{k^\ast} \cdot \exp\left(- \frac{n^{\{\beta^\ast\}}}{2\cdot (s(\alpha)+1)!}\right).$$
\end{proof}

Next we examine the expectation $\mathbb E(f_k)$ for $k\ge k^\ast$.

\begin{lemma} \label{lm106} For any $k\ge k^\ast$ one has
\begin{eqnarray}\label{expectationkk}
{\mathbb E}(f_k) = \frac{1}{(k+1)!} \cdot \left(\sum_{i=k}^r \frac{n^{\beta_i}}{(i-k)!}\right)\cdot (1+o(1)).
\end{eqnarray}
\end{lemma} 

\begin{proof} Note that for $k\ge k^\ast$ one has $k+1\ge k^\ast+1>\beta^\ast$ and hence for any $i=0, \dots, r$ we have 
$\beta_i< k+1$ which means that $i+1-\alpha_i<k+1$, i.e. $i-\alpha_i<k$. Next we observe that
$$q_i^{\binom {n-k}{i-k}} = 1 - \frac{n^{i-k-\alpha_i}}{(i-k)!}\cdot (1 + o(1)), \quad i=k, \dots, r,$$
since, as we mentioned above, all the exponents $i-k-\alpha_i$ are negative, and we may apply Remark \ref{rk1014}. Substituting this into (\ref{expectationk}) we obtain 
(\ref{expectationkk}). 
\end{proof}

\begin{corollary} \label{lm105} For any $k\ge k^\ast$ one has 
\begin{eqnarray}
\mathbb E(f_k) = \frac{1}{(k+1)!(i^\ast_k-k)!}\cdot n^{\beta^\ast_k}\cdot(1+o(1)),
\end{eqnarray}
where $$\beta^\ast_k=\max\{\beta_k, \beta_{k+1}, \dots, \beta_r\}$$ and $i^\ast_k$ is the unique integer $k\le i^\ast_k\le r$ such that 
$\beta_{i^\ast_k} = \beta^\ast_k$. 
\end{corollary}

\begin{proof} This follows automatically from Lemma \ref{lm106}. Here we also use our assumption (d) 
(saying that $\alpha_i-\alpha_j\notin \Z$) which guarantees uniqueness of the maximum. 
\end{proof}

Note that $\beta^\ast_{k^\ast} =\beta^\ast$ since $k^\ast \le i^\ast$. 
Besides,  $k^\ast<\beta^\ast<k^\ast+1$ and for $k>i^\ast$ one has $\beta_k^\ast<\beta^\ast$. 

\begin{theorem}\label{thm107}
Denoting the rate of exponential growth 
$$\gamma_k \, := \, \lim_{n\to \infty}\frac{\log \mathbb E(f_k)}{\log n} $$
we have 
$$\gamma_k = \gamma_k(\alpha) =\left\{
\begin{array}{lll}
k+1, &\mbox{for} & k<k^\ast, \\
\beta^\ast, &\mbox{for} & k^\ast\le k\le i^\ast,\\
\beta_k^\ast&\mbox{for}& k>i^\ast.\end{array}
\right.$$
In particular, the value of $\gamma_k$ is constant, maximal and is equal to $\beta^\ast$ for all $k$ satisfying $k^\ast\le k\le i^\ast$.
\end{theorem} 
\begin{proof} This follows from Corollary \ref{lm105} (for $k\ge k^\ast$). Besides, for $k< k^\ast$ we know that $f_k=\binom {n+1} {k+1}$, a.a.s. by Lemma \ref{lm101}. 
\end{proof}

The spread $s=s(\alpha)$ is the length of the flat maximum of the graph of the function $k\mapsto \gamma_k$.

%

We have 
$$\gamma_0< \gamma_1< \dots< \gamma_{k^\ast}= \dots= \gamma_{i^\ast}> \gamma_{k^\ast+1}\ge  \dots \ge \gamma_r$$

Note that in the case when the spread is zero, $k^\ast=i^\ast$, the sequence of exponents 
$$\gamma_0< \gamma_1< \dots<\gamma_{k^\ast-1} < \gamma_{k^\ast}> \gamma_{k^\ast+1}\ge  \dots \ge \gamma_r$$
is {\it unimodal}. 

\begin{example}{\rm 
Consider the case when $r=1$ (random graphs with respect to the upper probability measure). In this case we have two exponents $\alpha_0$ and $\alpha_1$. Recall that $\beta_0=1-\alpha_0$, $\beta_1=2-\alpha_1$, $\beta^\ast=\max\{\beta_0, \beta_1\}$ and $k^\ast =\lfloor \beta^\ast \rfloor$. 

We see that $k^\ast<0$ happens when $\alpha_0>1$ and $\alpha_1>2.$

We shall consider following three cases:

A) $k^\ast=0$ and $i^\ast=0$.

B) $k^\ast=0$ and $i^\ast=1$.

C) $k^\ast=1$ and $i^\ast=1$.

Case A) happens when $1-\alpha_0>2-\alpha_1$ and $1-\alpha_0>0$. This can be summarised by 
$\alpha_0<1$ and $\alpha_1>1+\alpha_0$.

Case B) can be characterised by the inequalities $0<\beta_0<\beta_1<1$ which can be rewritten as $\alpha_0<1$ and $1<\alpha_1<1+\alpha_0$.


In the case C) we have the inequalities: $1-\alpha_0<2-\alpha_1$ and $1<2-\alpha_1$. These inequalities reduce to the condition $0<\alpha_1<1$. 

Note that in cases A and C the spread is 0 and in case B the spread is 1. 
}
\end{example}

\begin{example}\label{spr} {\rm 
One can characterise the vectors $\alpha=(\alpha_0, \alpha_1, \dots, \alpha_r)$ with zero spread $s(\alpha)=0$ as follows. The index $i^\ast\in \{0, 1, \dots, r\}$ of the critical dimension satisfies 
$$\beta_{i^\ast}=\max\{\beta_0, \dots, \beta_r\}, \quad \lfloor \beta_{i^\ast}\rfloor = i^\ast.$$
In view of the definition $\beta_i=i+1-\alpha_i$ we see that $s(\alpha)=0$ is equivalent to 
\begin{eqnarray}\label{istar}\alpha_{i^\ast}<1, \quad \mbox{and}\quad \alpha_{i^\ast +k}>\alpha_{i^\ast} +k,\end{eqnarray}
for all $k=1, \dots, r-i^\ast$. 

From (\ref{istar}) we see that $i^\ast$ is the largest index satisfying $\alpha_{i^\ast}<1.$ However this condition alone is not sufficient. 

}
\end{example}

\begin{example}{\rm
Consider the case $r=2$ (two dimensional random simplicial complexes in the upper model). Using Example \ref{spr} we find that the vectors of exponents 
$\alpha=(\alpha_0, \alpha_1, \alpha_2)$ with zero spread are as follows:

$A_0).$ If the critical dimension is zero $k^\ast=0$ then $s=0$ is equivalent to $\alpha_0<1$ and $\alpha_1>1+\alpha_0$ and $\alpha_2>2+\alpha_0.$

$A_1).$ If the critical dimension is one, $k^\ast=1$, then $s=0$ is equivalent to $\alpha_1<1$ and $\alpha_2>1+\alpha_1$. 

$A_2).$ If the critical dimension $k^\ast=2$ then $s=0$ is equivalent to $\alpha_2<1.$

}
\end{example}

\section{Betti numbers in the upper model} 

\begin{theorem}\label{thmfk} Consider a random simplicial complex $Y\in \Omega_n^\ast$ with respect to the upper probability measure $\op$. 
We shall assume that the probability parameters $p_\sigma$ vanish for $\dim \sigma >r$ and for $\dim \sigma\le r$ they 
have the form $p_\sigma=n^{-\alpha_i}$, where $i=\dim \sigma$, and the exponents $\alpha_i>0$ satisfy the genericity assumptions (a) - (d), see \S \ref{101}. We shall also assume that the critical dimension $k^\ast\ge 0$ is non-negative. 
Let $f_k:\Omega_n^\ast\to \Z$ denote the random variable counting the number of $k$-dimensional faces of a random complex, where $k=0, 1, \dots, r$. Then there exists a sequence of real numbers $t_n\to 0$ such that for any $k\ge k^\ast$ one has
\begin{eqnarray}\hskip 1.5cm
\label{between2}
f_k(Y) \le (1+t_n)\cdot \frac{n^{\gamma_k(\alpha)}}{(k+1)!\cdot (i^\ast_k-k)!},
\end{eqnarray}
a.a.s. The exponent $\gamma_k(\alpha)$ is defined in Theorem \ref{thm107} and the integer $i^\ast_k\in \{k, k+1, \dots, r\}$ is defined by $\beta_{i^\ast_k}=\max\{\beta_k, \beta_{k+1}, \dots, \beta_r\}$. 

\end{theorem} 

Recall that $\beta_i$ denotes $i+1-\alpha_i$ and $k^\ast$ denotes the critical dimension, see Definition \ref{def102}.
\begin{proof}
Consider a random hypergraph $X\in \Omega_n$ with probability parameters 
$$p_\sigma = \left\{
\begin{array}{ll}
n^{-\alpha_i}, & \mbox{for}\quad \dim\sigma =i \le r, \\ \\

0, & \mbox{for}\quad \dim\sigma>r.
\end{array}\right.
$$
For $k=0, 1, \dots, r$ we shall denote by $g_k:\Omega_n\to \Z$ the random variable counting the number of $k$-dimensional faces. 
Note that $g_k$ is a binomial random variable ${\rm Bi}(\binom {n+1}{k+1}, n^{-\alpha_k})$ and hence we obviously have 
\begin{eqnarray}\label{43}
\E(g_k)= \binom {n+1} {k+1} \cdot n^{-\alpha_k}\, = \, \frac{n^{\beta_k}}{(k+1)!} \cdot (1+o(1)), \quad \quad k\le r.
\end{eqnarray} 
Thus, using the first moment method (the Markov inequality),  we see that for $\beta_k <0$ one has $g_k\equiv 0$, a.a.s. 

Note that our genericity assumptions of \S \ref{101} exclude the possibility $\beta_k=0$. 

Below we shall assume that $\beta_k> 0$. 

We may use the Chernoff bound (see  \cite{JLR}, Theorem 2.1), as we did earlier in \S \ref{sec9} to obtain
%
%
%
%
%

\begin{eqnarray}\label{between1}
(1-t_n) \cdot \frac{n^{\beta_k}}{(k+1)!}  \, \le\,  g_k \, \le\,  (1+t_n) \cdot \frac{n^{\beta_k}}{(k+1)!},
\end{eqnarray}
a.a.s., where $t_n\to 0$. 

Denote by $\Omega_{n, r}$ the set of hypergraphs $X\subset \Delta_n$ of dimension $\le r$. Similarly, denote by $\Omega^\ast_{n, r}$ the set of all simplicial subcomplexes $Y\subset\Delta_n$ of dimension $\le r$. We have the map 
$$\overline\mu_r: \Omega_{n, r}\to \Omega^\ast_{n, r}$$ 
which is the restriction of the map which appears in (\ref{maps}). Recall that for $X\in \Omega_{n, r}$ the simplicial complex $\overline \mu_r(X)$ is the minimal simplicial complex $Y$ containing $X$. In other words, $Y$ is obtained from $X$ by adding all faces of all simplexes of $X$. 

Since we assume that $p_\sigma=0$ for all simplexes $\sigma$ of dimension $>r$ we obtain that the measure $\Bbb P_n$ (given by (\ref{prob0})) is supported on $\Omega_{n, r}\subset \Omega_{n}$. Hence, we obtain that the upper measure $\op$ on $\Omega^\ast_{n, r}$ coincides with the direct image $(\overline \mu_r)_\ast(\Bbb P_n)$. 

For any $k=0, 1, \dots, r$ we have two random variables $g_k: \Omega_{n, r}\to \Z$ and $f'_k = f_k\circ \overline \mu_r: \Omega_{n, r}^\ast \to \Z$. From the structure of the map $\overline\mu_r$ we obtain the following inequality
\begin{eqnarray}
 f'_k \, \le \, \sum_{i=k}^r \binom {i+1}{k+1} g_i
\end{eqnarray}
and combining with (\ref{between1}) we find
\begin{eqnarray*}
f'_k \, \le \, (1+o(1))\cdot \frac{n^{\gamma_k}}{(k+1)!\cdot (i^\ast_k-k)!},
\end{eqnarray*}
a.a.s. By the definition, $\op=\overline {\mu_r}_\ast(\Bbb P_n)$, and hence the above inequality implies (\ref{between2}). 
\end{proof} 

Next we describe the Betti numbers of the random simplicial complexes in the upper model with extra 
assumption of vanishing of the spread $s(\alpha)=0$. 

\begin{theorem}\label{thmcrit1}\label{thmbetti2}  Consider a random simplicial complex $Y\in \Omega_n^\ast$ with respect to the upper probability measure $\op$. 
Assume that the probability parameters $p_\sigma$ vanish for $\dim \sigma >r$ and for $\dim \sigma\le r$ they 
have the form $p_\sigma=n^{-\alpha_i}$, where $i=\dim \sigma$, and the exponents $\alpha_i>0$ satisfy the genericity assumptions (a) - (d), see \S \ref{101}. We shall also assume that the critical dimension $k^\ast\ge 0$ is non-negative and the spread vanishes, 
$s(\alpha)=0.$
Then:

(i) the reduced Betti numbers in dimensions below the critical dimension vanish, i.e. $\tilde b_i(Y)=0,$ a.a.s. for $i<k^\ast$. 

(ii) the Betti number in the critical dimension $b_{k^\ast}(Y)$ dominates all other Betti numbers, a.a.s. 
More precisely, there exist constants $A, B>0$ such that 
\begin{eqnarray}\label{between3}
A\cdot n^{\beta^\ast} \le\,  b_{k^\ast}(Y) \le B\cdot n^{\beta^\ast},
\end{eqnarray}
a.a.s. and for any $j>k^\ast$ there exists $\epsilon_j>0$ such that 
\begin{eqnarray}\label{48}
b_j(Y) < n^{-\epsilon_j}\cdot b_{k^\ast}(Y), 
\end{eqnarray}
a.a.s.
\end{theorem}

%

\begin{proof} We noted in Example \ref{spr} that vectors $\alpha \- = \-(\alpha_0, \dots, \alpha_r)$ with the property that their spread is zero, $s(\alpha)=0$, are characterised by the inequalities 
\begin{eqnarray}\label{sprd0}
\alpha_{i^\ast}<1 \quad \mbox{and}\quad \alpha_{i^\ast+j}>\alpha_{i^\ast}+j,
\end{eqnarray}
for any $j=1, \dots, r-i^\ast$. Here $i^\ast$ is the index satisfying 
 $\beta_{i^\ast}= \max\{\beta_j; j=0, 1, \dots, r\}$. 
 
 Consider the pure $k$-dimensional random complex $Z_k$ with respect to the upper measure with probability parameter $p=n^{-\alpha}$ where 
 $$k=i^\ast= k^\ast\quad \mbox{and}\quad \alpha=\alpha_{i^\ast},$$ as we studied in \S \ref{sec8} above.  The complex $Z_k$ is naturally imbedded into $Y$. Since $\alpha\in (0,1)$ the results of \S \ref{sec9} are applicable. In particular, we see that $Z_k$ contains the full $(k-1)$-dimensional skeleton, and therefore the same is true with respect to $Y$. We know from Corollary \ref{lm93} that the reduced Betti numbers of $Z_k$ in dimensions $<k$ vanish. The complex $Y$ is obtained from $Z_k$ by adding simplexes of dimensions $\ge k$ and this clearly may not create cycles in dimension $<k$; thus, statement (i) follows. 
 
 Next we consider the Betti number $b_k(Y)$ in the critical dimension. Corollary \ref{cor92} gives an estimate
 \begin{eqnarray}\label{bky}
 b_k(Z_k) = (1+o(1))\cdot \frac{n^{k+1-\alpha}}{(k+1)!}.\end{eqnarray}
 The exponent $k+1-\alpha$ equals $\beta^\ast=\max\{\beta_0, \beta_1, \dots, \beta_r\}$. 
 Since $Z_k$ is a subcomplex of $Y^{(k)}$ the $k$-dimensional Betti number of $Y$ can be reduced when simplexes of 
 dimension $k+1$ are added. From the Morse inequalities we have 
 \begin{eqnarray}\label{morse}
  b_k(Z_k) - f_{k+1}(Y) \, \le\,  b_k(Y) \, \le\,  f_k(Y).\end{eqnarray}
  To explain the left inequality (\ref{morse}) we note that the exact sequence of homology groups
  $$H_{k+1}(Y, Z_k) \to H_k(Z_k) \to H_k(Y)$$
  (which is a part of the long exact sequence of the pair $(Y, Z_k)$) gives the inequality
  $$b_k(Y) \ge b_k(Z_k)-b_{k+1}(Y,Z_k)$$
  where (by the Morse inequality) we have $$b_{k+1}(Y, Z_k) \le f_{k+1}(Y, Z_k) = f_{k+1}(Y).$$
  Combining these two inequalities gives the left inequality of (\ref{morse}).

 We can use Theorem \ref{thmfk} to estimate $f_{k}(Y)$ and $f_{k+1}(Y)$ from above. Inequality (\ref{between2}) gives 
 \begin{eqnarray}\label{less}
 f_k(Y)\le C\cdot n^{\beta^\ast} \quad \mbox{and }\quad
 f_{k+1}(Y) \le C'\cdot n^{\beta^\ast_{k+1}},\end{eqnarray}
 a.a.s. Recall that the exponent $\beta^\ast_{k+1}$ is defined as $\max\{\beta_{k+1}, \dots, \beta_r\}$ which is strictly smaller than $\beta^\ast$. Combining the inequalities (\ref{bky}), (\ref{morse}) and (\ref{less}) we obtain
 \begin{eqnarray}
 A\cdot n^{\beta^\ast} \le b_k(Y) \le B\cdot n^{\beta^\ast},
 \end{eqnarray}
 a.a.s., 
 proving (\ref{between3}). 
 
 Finally, to prove (\ref{48}), consider an arbitrary $j>k$. Then $b_j(Y)\le f_j(Y)$ by the Morse inequality, and $f_j(Y)\le n^{\beta_j^\ast}$ by 
 (\ref{between2}). Here $\beta_j^\ast$ is defined as $\max\{\beta_j, \beta_{j+1}, \dots, \beta_r\}$. By definition, 
 we have $\beta_j^\ast< \beta^\ast$. Setting $\epsilon_j = (\beta^\ast -\beta^\ast_j)/2$ we have 
 $$b_j(Y) \le n^{\beta_j^\ast} < n^{-\epsilon_j} \cdot An^{\beta^\ast}\le n^{-\epsilon_j} \cdot b_k(Y),$$
 a.a.s.
This completes the proof. 
\end{proof}

\end{document}